\documentclass[a4paper,11pt]{amsart}

\usepackage[utf8]{inputenc}
\usepackage[T1]{fontenc}
\usepackage{lmodern}
\usepackage[english]{babel}
\usepackage[margin=32mm,bottom=40mm]{geometry}
\usepackage{amsmath,amsthm,amssymb}
\usepackage{mathtools}
\usepackage{mathdots}
\usepackage{microtype}
\usepackage{xcolor}

\definecolor{dblue}{rgb}{0,0,.6}
\usepackage[colorlinks=true,linkcolor=dblue,citecolor=dblue,filecolor=dblue,menucolor=dblue,urlcolor=dblue]{hyperref}

\newtheorem{theorem}{Theorem}
\newtheorem{corollary}[theorem]{Corollary}
\newtheorem{lemma}[theorem]{Lemma}
\newtheorem{proposition}[theorem]{Proposition}
\newtheorem{question}[theorem]{Question}

\newcommand*{\pr}{\ensuremath{\mathrm{pr}}}
\newcommand*{\OO}{\ensuremath{\mathcal{O}}}
\newcommand*{\PP}{\ensuremath{\mathbb{P}}}
\newcommand*{\CC}{\ensuremath{\mathbb{C}}}
\newcommand*{\ZZ}{\ensuremath{\mathbb{Z}}}
\newcommand*{\QQ}{\ensuremath{\mathbb{Q}}}

\newcommand*{\hodge}[1]{\mathmakebox[1em][c]{#1}}

\begin{document}   

\title{The construction problem for Hodge numbers modulo an integer}
\author{Matthias Paulsen}
\address{Mathematisches Institut \\ Ludwig-Maximilians-Universität München \\ Theresienstr.~39 \\ D-80333~München \\ Germany}
\email{paulsen@math.lmu.de}
\author{Stefan Schreieder}
\address{Mathematisches Institut \\ Ludwig-Maximilians-Universität München \\ Theresienstr.~39 \\ D-80333~München \\ Germany}
\email{schreieder@math.lmu.de}
\date{August 2, 2019}

\begin{abstract}
For any integer $m\ge2$ and any dimension $n\ge1$, we show that any $n$-dimensional Hodge diamond with values in $\ZZ/m\ZZ$
is attained by the Hodge numbers of an $n$-dimensional smooth complex projective variety.
As a corollary, there are no polynomial relations among the Hodge numbers of $n$-dimensional smooth complex projective varieties
besides the ones induced by the Hodge symmetries, which answers a question raised by Koll\'ar in 2012.
\end{abstract}

\maketitle

\section{Introduction} 

Hodge theory allows one to decompose the $k$-th Betti cohomology of an $n$-dimensional compact Kähler manifold $X$ into its $(p,q)$-pieces for all $0\le k\le 2n$:
\[ H^k(X,\CC) = \bigoplus_{\substack{p+q=k \\ 0\le p,q\le n}}H^{p,q}(X) \;,\quad \overline{H^{p,q}(X)}=H^{q,p}(X) \;. \]
The $\CC$-linear subspaces $H^{p,q}(X)$ are naturally isomorphic to the Dolbeault cohomology groups $H^q(X,\Omega_X^p)$.

The integers $h^{p,q}(X)=\dim_\CC H^{p,q}(X)$ for $0\le p,q\le n$ are called Hodge numbers. One usually arranges them in the so called Hodge diamond:
\[ \begin{matrix}
&&&& \hodge{h^{n,n}} &&&& \\
&&& \hodge{h^{n,n-1}} && \hodge{h^{n-1,n}} &&& \\
&& \hodge{\iddots} && \hodge{\vdots} && \hodge{\ddots} && \\
& \hodge{h^{n,1}} &&& &&& \hodge{h^{1,n}} & \\
\hodge{h^{n,0}} && \hodge{h^{n-1,1}} && \hodge{\cdots} && \hodge{h^{1,n-1}} && \hodge{h^{0,n}} \\
& \hodge{h^{n-1,0}} &&& &&& \hodge{h^{0,n-1}} & \\
&& \hodge{\ddots} && \hodge{\vdots} && \hodge{\iddots} && \\
&&& \hodge{h^{1,0}} && \hodge{h^{0,1}} &&& \\
&&&& \hodge{h^{0,0}} &&&&
\end{matrix} \]
The sum of the $k$-th row of the Hodge diamond equals the $k$-th Betti number.
We always assume that a Kähler manifold is compact and connected, so we have $h^{0,0}=h^{n,n}=1$.

Complex conjugation and Serre duality induce the symmetries
 \begin{align} \label{eq:HS}
h^{p,q}=h^{q,p}=h^{n-p,n-q}\ \ \text{for all}\ \ 0\le p,q\le n\; .
\end{align}
Additionally, we have the Lefschetz inequalities
\begin{align}\label{eq:Lineq}
h^{p,q}\le h^{p+1,q+1}\ \ \text{for}\ \ p+q<n\; .
\end{align}

While Hodge theory places severe restrictions on the geometry and topology of Kähler manifolds,
Simpson points out in \cite{simpson} that very little is known to which extent the theoretically possible phenomena actually occur.
This leads to the following construction problem for Hodge numbers:

\begin{question}\label{question:construction}
Let $(h^{p,q})_{0\le p,q\le n}$ be a collection of non-negative integers with $h^{0,0}=1$ obeying the Hodge symmetries \eqref{eq:HS} and the Lefschetz
inequalities \eqref{eq:Lineq}.
Does there exist a Kähler manifold $X$ such that $h^{p,q}(X)=h^{p,q}$ for all $0\le p,q\le n$?
\end{question}

After results in dimensions two and three (see e.\,g.\ \cite{hunt}), significant progress has been made by the second author \cite{schreieder}.
For instance, it is shown in \cite[Theorem~3]{schreieder} that the above construction problem is fully solvable for large parts of the Hodge diamond in arbitrary dimensions.
In particular, the Hodge numbers in a given weight $k$ may be arbitrary (up to a quadratic lower bound on $h^{p,p}$ if $k=2p$ is even) and so the outer Hodge numbers can be far larger than the inner Hodge numbers (see \cite[Theorem~1]{schreieder}), contradicting earlier expectations formulated in \cite{simpson}.
Weaker results with simpler proofs, concerning the possible Hodge numbers in a given weight, have later been obtained by Arapura \cite{arapura}.

In \cite{schreieder}, it was also observed that one cannot expect a positive answer to Question~\ref{question:construction} in its entirety.
For example, any $3$-dimensional Kähler manifold $X$ with $h^{1,1}(X)=1$ and $h^{2,0}(X)\ge1$ satisfies $h^{2,1}(X)<12^6\cdot h^{3,0}(X)$, see \cite[Proposition~28]{schreieder}.
Therefore, a complete classification of all possible Hodge diamonds of Kähler manifolds or smooth complex projective varieties
seems hopelessly complicated.

While these inequalities aggravate the construction problem for Hodge numbers,
one might ask whether there also exist number theoretic obstructions for possible Hodge diamonds.
For example, the Chern numbers of Kähler manifolds satisfy certain congruences due to integrality conditions
implied by the Hirzebruch--Riemann--Roch theorem.

For an arbitrary integer $m\ge2$, let us consider the Hodge numbers of a Kähler manifold in $\ZZ/m\ZZ$, which forces all inequalities to disappear.
The purpose of this paper is to show that Question~\ref{question:construction} is modulo~$m$ completely solvable
even for smooth complex projective varieties: 

\begin{theorem}\label{theorem:main}
Let $m\ge2$ be an integer.
For any integer $n\ge1$ and any collection of integers $(h^{p,q})_{0\le p,q\le n}$ such that $h^{0,0}=1$ and $h^{p,q}=h^{q,p}=h^{n-p,n-q}$ for $0\le p,q\le n$,
there exists a smooth complex projective variety $X$ of dimension~$n$ such that \[ h^{p,q}(X)\equiv h^{p,q}\pmod m \] for all $0\le p,q\le n$.
\end{theorem}

Therefore, the Hodge numbers of Kähler manifolds do not follow any number theoretic rules, and the behaviour
of smooth complex projective varieties is the same in this aspect.

As a consequence of Theorem~\ref{theorem:main}, we show:

\begin{corollary}\label{corollary:poly}
Up to the Hodge symmetries \eqref{eq:HS}, there are no polynomial relations among the Hodge numbers of smooth complex projective varieties of the same dimension.
\end{corollary}

In particular, there are no polynomial relations in the strictly larger
class of Kähler manifolds, which was a question raised by Koll\'ar after a colloquium talk of Kotschick at the University of Utah in fall 2012.
For linear relations among Hodge numbers, this question was settled in work of Kotschick and the second author in \cite{kotschick-schreieder}.

We call the Hodge numbers $h^{p,q}(X)$ with $p\in\{0,n\}$ or $q\in\{0,n\}$ (i.\,e.\ the ones placed on the border of
the Hodge diamond) the \emph{outer Hodge numbers} of $X$ and the remaining ones the \emph{inner Hodge numbers}.
Note that the outer Hodge numbers are birational invariants and are thus determined by the birational equivalence class of $X$.

Our proof shows (see Theorem~\ref{theorem:inner} below) that any smooth complex projective variety is birational
to a smooth complex projective variety with prescribed inner Hodge numbers in $\ZZ/m\ZZ$.
As a corollary, there are no polynomial relations among the inner Hodge numbers within a given birational equivalence class.
This is again a generalization of the corresponding result for linear relations obtained in \cite[Theorem~2]{kotschick-schreieder}.

The proof of Theorem~\ref{theorem:main} can thus be divided into two steps:
First we solve the construction problem modulo~$m$ for the outer Hodge numbers.
This is done in Section~\ref{section:outer}.
Then we show the aforementioned result that the inner Hodge numbers can be adjusted arbitrarily in $\ZZ/m\ZZ$
via birational equivalences (in fact, via repeated blow-ups).
This is done in Section~\ref{section:inner}.
Finally, in Section~\ref{section:poly} we deduce that no non-trivial polynomial relations between Hodge numbers exist,
thus answering Koll\'ar's question.

\section{Outer Hodge numbers}\label{section:outer}

We prove the following statement via induction on the dimension $n\ge1$.
\begin{proposition}\label{proposition:outer}
For any collection of integers $(h^{p,0})_{1\le p\le n}$, there exists a smooth complex projective variety $X_n$ of dimension~$n$
together with a very ample line bundle $L_n$ on $X_n$ such that \[ h^{p,0}(X_n)\equiv h^{p,0}\pmod m \] for all $1\le p\le n$ and
\[ \chi(L_n^{-1})\equiv1\pmod m \;. \]
\end{proposition}
\begin{proof}
We take $X_1$ to be a curve of genus~$g$ where $g\equiv h^{1,0}\pmod m$.
Further, we take $L_1$ to be a line bundle of degree~$d$ on $X_1$ where $d>2g$ and $d\equiv -g\pmod m$.
Then $L_1$ is very ample and by the Riemann--Roch theorem we have $\chi(L_1^{-1})\equiv1\pmod m$.

Now let $n>1$.
We define a collection of integers $(k^{p,0})_{-1\le p\le n-1}$ recursively via
\[ k^{-1,0}=0\;,\quad k^{0,0}=1\;,\quad k^{p,0}=h^{p,0}-2k^{p-1,0}-k^{p-2,0} \text{ for } 1\le p\le n-1 \;. \]
We choose $X_{n-1}$ and $L_{n-1}$ by induction hypothesis such that $h^{p,0}(X_{n-1})\equiv k^{p,0}\pmod m$ for all $1\le p\le n-1$.

Let $E$ be a smooth elliptic curve and let $L$ be a very ample line bundle of degree~$d$ on $E$ such that $d\equiv1\pmod m$.
Let $e$ be a positive integer such that \[ e\equiv1+\sum_{p=1}^n(-1)^ph^{p,0}\pmod m \;. \]
Let $X_n\subset X_{n-1}\times E\times E$ be a hypersurface defined by a general section of the very ample line bundle
\[ P_n = \pr_1^*L_{n-1}\otimes\pr_2^*L^{m-1}\otimes\pr_3^*L^e \] on $X_{n-1}\times E\times E$.
By Bertini's theorem, we may assume $X_n$ to be smooth and irreducible.
Let $L_n$ be the restriction to $X_n$ of the very ample line bundle
\[ Q_n = \pr_1^*L_{n-1}\otimes\pr_2^*L\otimes\pr_3^*L \] on $X_{n-1}\times E\times E$.
Then $L_n$ is again very ample.

By the Lefschetz hyperplane theorem, we have
\[ h^{p,0}(X_n)=h^{p,0}(X_{n-1}\times E\times E) \] 
for all $1\le p\le n-1$. Since the Hodge diamond of $E\times E$ is
\[ \begin{matrix} &&1&& \\ &2&&2& \\ 1&&4&&1 \\ &2&&2& \\ &&1&& \end{matrix} \;, \]
Künneth's formula yields
\begin{align*}
h^{p,0}(X_n) &= h^{p,0}(X_{n-1})+2h^{p-1,0}(X_{n-1})+h^{p-2,0}(X_{n-1}) \\
&\equiv k^{p,0}+2k^{p-1,0}+k^{p-2,0} \pmod m \\
&= h^{p,0}
\end{align*}
for all $1\le p\le n-1$. Therefore, it only remains to show that $h^{n,0}(X_n)\equiv h^{n,0}\pmod m$ and $\chi(L_n^{-1})\equiv1\pmod m$.
Since \[ \chi(\OO_{X_n})=1+\sum_{p=1}^n(-1)^ph^{p,0}(X_n) \;, \] the congruence $h^{n,0}(X_n)\equiv h^{n,0}\pmod m$
is equivalent to $\chi(\OO_{X_n})\equiv e\pmod m$.

By definition of $X_n$, the ideal sheaf on $X_{n-1}\times E\times E$ of regular functions vanishing on $X_n$
is isomorphic to the sheaf of sections of the dual line bundle $P_n^{-1}$. Hence, there is a short exact sequence
\begin{equation}
0 \to P_n^{-1} \to \OO_{X_{n-1}\times E\times E} \to i_*\OO_{X_n} \to 0 \label{eq:ses}
\end{equation}
of sheaves on $X_{n-1}\times E\times E$ where $i\colon X_n\to X_{n-1}\times E\times E$ denotes the inclusion.
Together with Künneth's formula and the Riemann--Roch theorem, we obtain
\begin{align*}
\chi(\OO_{X_n}) &= \chi(\OO_{X_{n-1}\times E\times E}) - \chi(P_n^{-1}) \\
&= \chi(\OO_{X_{n-1}})\underbrace{\chi(\OO_E)^2}_{=0} - \underbrace{\chi(L_{n-1}^{-1})}_{\equiv1}\underbrace{\chi(L^{1-m})}_{\equiv1}\underbrace{\chi(L^{-e})}_{\equiv-e} \\
&\equiv e \pmod m \;.
\end{align*}
Tensoring \eqref{eq:ses} with $Q_n^{-1}$ yields the short exact sequence
\[ 0 \to P_n^{-1}\otimes Q_n^{-1} \to Q_n^{-1} \to i_*i^*Q_n^{-1} \to 0 \]
and thus
\begin{align*}
\chi(L_n^{-1}) &= \chi(Q_n^{-1}) - \chi(P_n^{-1}\otimes Q_n^{-1}) \\
&= \underbrace{\chi(L_{n-1}^{-1})}_{\equiv1}\underbrace{\chi(L^{-1})^2}_{\equiv1} - \chi(L_{n-1}^{-2})\underbrace{\chi(L^{-m})}_{\equiv0}\chi(L^{-e-1}) \\
&\equiv 1 \pmod m \;.
\end{align*}
This finishes the induction step.
\end{proof}

\section{Inner Hodge numbers}\label{section:inner}

We now show the following result, which significantly improves \cite[Theorem~2]{kotschick-schreieder}.

\begin{theorem}\label{theorem:inner}
Let $X$ be a smooth complex projective variety of dimension~$n$ and let $(h^{p,q})_{1\le p,q\le n-1}$ be any collection of integers
such that $h^{p,q}=h^{q,p}=h^{n-p,n-q}$ for $1\le p,q\le n-1$. Then $X$ is birational to a smooth complex projective variety $X'$
such that \[ h^{p,q}(X')\equiv h^{p,q}\pmod m \] for all $1\le p,q\le n-1$.
\end{theorem}
Together with Proposition~\ref{proposition:outer}, this will complete the proof of Theorem~\ref{theorem:main}.

Let us recall the following result on blow-ups, see e.\,g.\ \cite[Theorem~7.31]{voisin}:
If $\widetilde X$ denotes the blow-up of a Kähler manifold $X$ along a closed submanifold $Z\subset X$ of codimension~$c$,
we have
\[ H^{p,q}(\widetilde X) \cong H^{p,q}(X) \oplus \bigoplus_{i=1}^{c-1}H^{p-i,q-i}(Z) \;. \]
Therefore,
\begin{equation}
h^{p,q}(\widetilde X) = h^{p,q}(X) + \sum_{i=1}^{c-1}h^{p-i,q-i}(Z) \;. \label{eq:blowup}
\end{equation}

In order to prove Theorem~\ref{theorem:inner}, we first show that we may assume that $X$ contains certain subvarieties,
without modifying its Hodge numbers modulo~$m$.

\begin{lemma}\label{lemma:subvariety}
Let $X$ be a smooth complex projective variety of dimension~$n$.
Let $r,s\ge0$ be integers such that $r+s\le n-1$.
Then $X$ is birational to a smooth complex projective variety $X'$ of dimension~$n$
such that $h^{p,q}(X')\equiv h^{p,q}(X)\pmod m$ for all $0\le p,q\le n$ and such that $X'$ contains at least $m$ disjoint smooth closed subvarieties that are all isomorphic to a projective bundle of rank~$r$ over $\PP^s$.
\end{lemma}
\begin{proof}
We first blow up $X$ in a point and denote the result by $\widetilde X$.
The exceptional divisor is a subvariety in $\widetilde X$ isomorphic to $\PP^{n-1}$.
In particular, $\widetilde X$ contains a copy of $\PP^s\subset\PP^{n-1}$. Now we blow up $\widetilde X$ along $\PP^s$ to obtain $\widehat X$.
The exceptional divisor in $\widehat X$ is the projectivization of the normal bundle of $\PP^s$ in $\widetilde X$.
Since $\PP^s$ is contained in a smooth closed subvariety of dimension~$r+s+1$ in $\widetilde X$ (choose either $\PP^{r+s+1}\subset\PP^{n-1}$ if $r+s<n-1$ or $\widetilde X$ if $r+s=n-1$),
the normal bundle of $\PP^s$ in $\widetilde X$ contains a vector subbundle of rank~$r+1$, and hence its projectivization contains a projective subbundle of rank~$r$.
Therefore, $\widehat X$ admits a subvariety isomorphic to the total space of a projective bundle of rank~$r$ over $\PP^s$.

By \eqref{eq:blowup}, the above construction only has an additive effect on the Hodge diamond,
i.\,e.\ the differences between respective Hodge numbers of $\widehat X$ and $X$ are constants independent of $X$.
Hence, we may apply the above construction $m-1$ more times to obtain a smooth complex projective variety $X'$ containing $m$ disjoint copies of
the desired projective bundle and satisfying $h^{p,q}(X')\equiv h^{p,q}(X)\pmod m$.
\end{proof}

In the following, we consider the primitive Hodge numbers \[ l^{p,q}(X)=h^{p,q}(X)-h^{p-1,q-1}(X) \] for $p+q\le n$.
Clearly, it suffices to show Theorem~\ref{theorem:inner} for a given collection $(l^{p,q})_{(p,q)\in I}$ of primitive Hodge numbers instead,
where \[ I = \{(p,q) \mid 1\le p\le q\le n-1\text{ and }p+q\le n \} \;. \]
This is because one can get back the original Hodge numbers from the primitive Hodge numbers via the relation
\[ h^{p,q}(X) = h^{0,q-p}(X) + \sum_{i=1}^p l^{i,q-p+i}(X) \]
for $p\le q$ and $p+q\le n$, and $h^{0,q-p}(X)$ is a birational invariant.

We define a total order $\prec$ on $I$ via
\[ (r,s)\prec(p,q) \iff r+s<p+q \text{ or }(r+s=p+q \text{ and }s<q) \;. \]

\begin{proposition}\label{proposition:incr}
Let $X$ be a smooth complex projective variety of dimension~$n$.
Let $(r,s)\in I$.
Then $X$ is birational to a smooth complex projective variety $X'$ of dimension~$n$
such that \[ l^{r,s}(X')\equiv l^{r,s}(X)+1\pmod m \] and \[ l^{p,q}(X')\equiv l^{p,q}(X)\pmod m \] for all $(p,q)\in I$ with $(r,s)\prec(p,q)$.
\end{proposition}
\begin{proof}
By Lemma~\ref{lemma:subvariety}, we may assume that $X$ contains $m$ disjoint copies of a projective bundle of rank~$r-1$ over $\PP^{s-r+1}$.
Therefore, it is possible to blow up $X$ along a projective bundle $B_d$ of rank~$r-1$ over smooth hypersurfaces $Y_d\subset\PP^{s-r+1}$ of degree~$d$ (in case of $r=s$, $Y_d$ just consists of $d$ distinct points in $\PP^1$) and we may repeat this procedure $m$~times and with different values for $d$.
The Hodge numbers of $B_d$ are the same as for the trivial bundle $Y_d\times\PP^{r-1}$, see e.\,g.\ \cite[Lemma~7.32]{voisin}.

By the Lefschetz hyperplane theorem, the Hodge diamond of $Y_d$ is the sum of the Hodge diamond of $Y_1\cong\PP^{s-r}$, having non-zero entries
only in the middle column, and of a Hodge diamond depending on $d$, having non-zero entries only in the middle row.
It is well known (e.\,g.\ by computing Euler characteristics as in Section~\ref{section:outer}) that the two outer entries of this middle row are precisely $\binom{d-1}{s-r+1}$.

Now we blow up $X$ once along $B_{s-r+2}$ and $m-1$~times along $B_1$ and denote the resulting smooth complex projective variety by $X'$.
Due to \eqref{eq:blowup} and Künneth's formula, this construction affects the Hodge numbers modulo~$m$ in the same way as
if we would blow up a single subvariety $Z\times\PP^{r-1}\subset X$, where $Z$ is a (formal) $(s-r)$-dimensional Kähler manifold
whose Hodge diamond is concentrated in the middle row and has outer entries equal to $\binom{s-r+2-1}{s-r+1}=1$.
In particular, we have $h^{p,q}(Z\times\PP^{r-1})=0$ unless $s-r\le p+q\le s+r-2$ (and $p+q$ has the same parity as $s-r$) and $|p-q|\le s-r$.
On the other hand, $h^{p,q}(Z\times\PP^{r-1})=1$ if $s-r\le p+q\le s+r-2$ and $|p-q|=s-r$.

Taking differences in \eqref{eq:blowup}, it follows that
\[ l^{p,q}(X') \equiv l^{p,q}(X) + h^{p-1,q-1}(Z\times\PP^{r-1}) - h^{p-n+s-1,q-n+s-1}(Z\times\PP^{r-1}) \pmod m\]
for all $p+q\le n$. But we have \[ (p-n+s-1)+(q-n+s-1) = p+q-2n+2s-2 \le 2s-n-2 \le s-r-2 \] and hence $h^{p-n+s-1,q-n+s-1}(Z\times\PP^{r-1})=0$
for all $(p,q)\in I$ by the above remark.

Further,
\[ l^{r,s}(X') \equiv l^{r,s}(X) + h^{r-1,s-1}(Z\times\PP^{r-1}) = l^{r,s}(X) + 1 \pmod m \]
since $s-r\le(r-1)+(s-1)\le s+r-2$ and $|r-s|=s-r$.

Finally, $r+s<p+q$ implies $(p-1)+(q-1)>s+r-2$, while $r+s=p+q$ and $s<q$ imply $|p-q|>s-r$, so we have $h^{p-1,q-1}(Z\times\PP^{r-1})=0$ in both cases and thus
\[ l^{p,q}(X') \equiv l^{p,q}(X) + h^{p-1,q-1}(Z\times\PP^{r-1}) = l^{p,q}(X) \pmod m \]
for all $(p,q)\in I$ with $(r,s)\prec(p,q)$.
\end{proof}

\begin{proof}[Proof of Theorem~\ref{theorem:inner}]
The statement is an immediate consequence of applying Proposition~\ref{proposition:incr} inductively $t_{p,q}$~times to each $(p,q)\in I$ in the descending order induced by $\prec$,
where $t_{p,q}\equiv l^{p,q}-l^{p,q}(X_{p,q})\pmod m$ and $X_{p,q}$ is the variety obtained in the previous step.
\end{proof}

\section{Polynomial relations}\label{section:poly}

The following principle seems to be classical.
\begin{lemma}\label{lemma:modulo}
Let $N\ge1$ and let $S\subset\ZZ^N$ be a subset such that its reduction modulo~$m$ is the whole of $(\ZZ/m\ZZ)^N$
for infinitely many integers $m\ge2$. If $f\in\CC[x_1,\ldots,x_N]$ is a polynomial vanishing on $S$, then $f=0$.
\end{lemma}
\begin{proof}
Let $f\in\CC[x_1,\ldots,x_N]$ be a non-zero polynomial vanishing on $S$.
By choosing a $\QQ$-basis of $\CC$ and a $\QQ$-linear projection $\CC\to\QQ$ which sends a non-zero coefficient of $f$ to $1$,
we see that we may assume that the coefficients of $f$ are rational, hence even integral.
Since $f\ne0$, there exists a point $z\in\ZZ^N$ such that $f(z)\ne0$.
Choose an integer $m\ge2$ from the assumption which does not divide $f(z)$.
Then $f(z)\not\equiv0\pmod m$. However, we have $z\equiv s\pmod m$ for some $s\in S$ and thus $f(z)\equiv f(s)=0\pmod m$,
because $f\in\ZZ[x_1,\ldots,x_N]$. This is a contradiction.
\end{proof}

\begin{proof}[Proof of Corollary~\ref{corollary:poly}]
This follows by applying Lemma~\ref{lemma:modulo} to the set $S$ of possible Hodge diamonds,
where we consider only a non-redundant quarter of the diamond to take the Hodge symmetries into account.
Theorem~\ref{theorem:main} guarantees that the reductions of $S$ modulo~$m$ are surjective
even for all integers $m\ge2$.
\end{proof}

In the same way Theorem~\ref{theorem:main} implies Corollary \ref{corollary:poly}, Theorem~\ref{theorem:inner} yields the following.

\begin{corollary}\label{corollary:poly:birational}
There are no non-trivial polynomial relations among the inner Hodge numbers of all smooth complex projective varieties in any given birational equivalence class.
\end{corollary}

\section*{Acknowledgements}
The second author thanks  J.\ Koll\'ar and D.\ Kotschick for independently making him aware of Koll\'ar's question (answered in Corollary~\ref{corollary:poly} above), which was the starting point of this paper.
The authors are grateful to the referees for useful suggestions.
This work is supported by the DFG project ``Topologische Eigenschaften von algebraischen Varietäten'' (project no.\ 416054549).

\bibliographystyle{myamsalpha}
\bibliography{biblio}

\providecommand{\bysame}{\leavevmode\hbox to3em{\hrulefill}\thinspace}
\providecommand{\MR}{\relax\ifhmode\unskip\space\fi MR }
\providecommand{\MRhref}[2]{%
  \href{http://www.ams.org/mathscinet-getitem?mr=#1}{#2}
}
\providecommand{\href}[2]{#2}
\begin{thebibliography}{Hun89}

\bibitem[Ara16]{arapura}
D.~Arapura, \emph{Geometric {Hodge} structures with prescribed {Hodge}
  numbers}, Recent Advances in Hodge Theory (M.~Kerr and G.~Pearlstein, eds.),
  London Mathematical Society Lecture Note Series, no. 427, Cambridge
  University Press, 2016, pp.~414--421.

\bibitem[Hun89]{hunt}
B.~Hunt, \emph{Complex manifold geography in dimension 2 and 3}, Journal of
  Differential Geometry \textbf{30} (1989), 51--153.

\bibitem[KS13]{kotschick-schreieder}
D.~Kotschick and S.~Schreieder, \emph{The {Hodge} ring of {Kähler} manifolds},
  Compositio Mathematica \textbf{149} (2013), 637--657.

\bibitem[Sch15]{schreieder}
S.~Schreieder, \emph{On the construction problem for {Hodge} numbers}, Geometry
  \& Topology \textbf{19} (2015), 295--342.

\bibitem[Sim04]{simpson}
C.~Simpson, \emph{The construction problem in {Kähler} geometry}, Different
  Faces of Geometry (S.~Donaldson, Y.~Eliashberg, and M.~Gromov, eds.),
  International Mathematical Series, vol.~3, Springer, 2004, pp.~365--402.

\bibitem[Voi03]{voisin}
C.~Voisin, \emph{{Hodge} theory and complex algebraic geometry {I}}, Cambridge
  studies in advanced mathematics, no.~76, Cambridge University Press, 2003.

\end{thebibliography}

\end{document}